\documentclass[12pt]{amsart}
    
    \newtheorem{theorem}{Theorem}[section]
    
    \newtheorem{proposition}[theorem]{Proposition}
    \newtheorem{corollary}[theorem]{Corollary}

\newtheorem{lem}[theorem]{Lemma}
    \theoremstyle{definition}
    
   \newtheorem{defn}[theorem]{Definition}
    \newtheorem{example}[theorem]{Example}
    \newtheorem{remark}[theorem]{Remark}

\def\GKdim{\operatorname{GK dim}}

    \theoremstyle{remark}
    \usepackage{amscd,amssymb}

\begin{document}
    \baselineskip 17 pt

    \title[A short proof that the free
associative algebra is Hopfian]
    {A short proof that the free
associative algebra is Hopfian}

    \author[A. Belov-Kanel, Louis Rowen and Jie-Tai Yu]
    {Alexei Belov-Kanel, Louis Rowen and Jie-Tai Yu}
    \address{Department of Mathematics, Bar-Ilan University,
    Ramat-Gan 52900, Israel}
    \email{rowen@math.biu.ac.il}
    \address{ College of Mathematics, Shenzhen University, Shenzhen, China}
    \email{jietaiyu@szu.edu.cn,\
    yujietai@yahoo.com, kanelster@gmail.com}

    \thanks
    {The research of A. Belov-Kanel  and L.~Rowen was partially supported
    by the Israel Science Foundation  (grant no. 1207/12).}
    \thanks{The authors thank L. Small for many useful references}
    \thanks{The research of Jie-Tai Yu was supported by a Grant from National 1000 Plan.}



    \begin{abstract} A short proof is given of the fact that
    for various classes of algebras including the free associative
    algebra are Hopfian, i.e., every epimorphism is an automorphism. This further simplifies
 the   Dicks-Lewin solution of the Jacobian conjecture for the free associative
    algebra. It also relates to the Jacobian conjecture for
    meta-abelian algebras.
    \end{abstract}

    \maketitle

    \section{Introduction and main results}

    \smallskip

The main object of this note is to use PI-theory to simplify the
results of Dicks and Lewin~\cite{DL} on the automorphisms of the
free
algebra $F\{ X\}$, namely that 
if the Jacobian is invertible, then every
    endomorphism is an automorphism. Their
    proof is in two parts: Every epimorphism
 $F\{ X\} \to F\{ X\}$ is an isomorphism, and if the Jacobian is invertible, then every
    endomorphism is an epimorphism. We follow this route, and see
    how these arguments apply to wider classes of rings. In the appendix,
    we tie this in with the Jacobian conjecture.

    \section{Hopfian rings}
\begin{defn} An algebra $R$ is \textbf{Hopfian} if every epimorphism (i.e., onto
algebra homomorphism) $R \to R$ is an isomorphism.
 \end{defn}

Dicks and Lewin reduced the Jacobian conjecture for $F\{ X\} $ to
the question of whether $F\{ X\} $ is Hopfian, and proved it (using
some deep theory) for the free algebra in two variables. In fact,
this had already been resolved for any finite set of variables by
Orzech and Ribes~\cite{OR,O}, with a rather direct proof given in
\cite{C}.

 In this note,  relying on
considerations of growth, we give a
 quick and easy proof of a more general result (based on facts about the Gelfand-Kirillov dimension)
 which implies at once that the free associative algebra $F\{ X\} $ is
 Hopfian.

Recall that the \textbf{Gelfand-Kirillov dimension}  $\GKdim (A)$
of an affine algebra $A = F\{ a_1, \dots, a_\ell \}$ is
\begin{equation} \label{GKdim}
\GKdim (A): = \underset{n \to \infty}\varlimsup\log _n \tilde d_n,
\end{equation}
where  $A_n = \sum F a_{i_1}\cdots a_{i_n}$ and  $d_n = \dim _F
A_n$.

The standard reference on Gelfand-Kirillov dimension is \cite{KL}.
Although the $d_n$ depend on the choice of the generating set $a_1,
\dots, a_\ell$,  $\GKdim (A)$ is independent of the choice of the
generating set.  Let us tighten this fact a bit: Suppose that  $A' =
F\{ a_1', \dots, a'_\ell \}$ and $d'_n = \dim _F A'_n$. We say that
the growth rate  of the $d_n$  is less than or equal to the  growth
rate of the $d'_n$ if there are constants $c,k $ such that $d_n' \le
c d_{kn}$. This defines an equivalence, and it is easy to see that
the growth rate of $A$ with respect to any two sets of generators is
the same.

\begin{lem}\label{easylem} Suppose $R$ is an affine algebra in which
the growth of $R/I$ is less than the growth of $R$, for each ideal
$I$ of $R$.  Then $R$ is Hopfian.

In particular, if $\GKdim (R/I)< \GKdim(R)$ for all   ideals $I$
of $R$, then $R$ is Hopfian.
\end{lem}
\begin{proof} For any epimorphism $\varphi :R \to R,$ one has $\varphi
(R) \cong R/ \ker \varphi,$ but then $\varphi (R)$ and $R$ have
the same growth rates, implying $ \ker \varphi = 0.$
\end{proof}

The   hypothesis of Lemma~\ref{easylem} holds for prime
PI-algebras, cf.~\cite[Theorem~11.2.12]{BR}, so we have:

\begin{corollary} Any prime affine PI-algebra is Hopfian.\end{corollary}

On the other hand, \cite{AFS} provides an affine PI-algebra that  is
not Hopfian.

\begin{example} $R$ and $R/I$ could have different growth rates
even if $\GKdim (R/I)=\GKdim(R)$. For example, let $R$ be the
subalgebra of the free associative algebra generated by all subwords
of $u_n$ for any $n$, where $u_1 = xyx$ and $u_{n+1} =
x^{10^n}u^nx^{10^n}yx^{10^n}u_nx^{10^n},$ a prime algebra, of
$\GKdim  2,$ and $I$ be the ideal generated by all words of degree 2
in $y$. Then  $\GKdim (R/I) = 2,$ although the growth rate of $R/I$
is less than that of $R$. This example is not a PI-algebra.
\end{example}

A \textbf{$T$-ideal} of an ideal $R$ is an ideal invariant under
all ring endomorphisms.

\begin{lem}\label{easierlem}If $\mathcal I$ is a $T$-ideal of $R$, then any endomorphism $\varphi$ of $R$
induces an endomorphism of $R/\mathcal I.$\end{lem}
\begin{proof} Define $\varphi : R/\mathcal I \to R/\mathcal I$ by
$\varphi (a + \mathcal I) = \varphi (a) + \mathcal I.$ This is
well-defined since $\varphi(\mathcal I) \subseteq \mathcal I$ by
hypothesis.
\end{proof}

We say that $R$ is \textbf{$T$-residually Hopfian} if the
intersection of those $T$-ideals $I$ of $R$ for which $R/I$ is
Hopfian is $0.$

The following result, whose proof follows that of
\cite[Theorem~7]{BLK}, attributed to Markov, unifies instances of
Hopfian algebras.

\begin{proposition}\label{resH} Any  $T$-residually Hopfian  algebra is  Hopfian.
  \end{proposition}
\begin{proof}   Let $\varphi: R \to R$ be an epimorphism, with some nonzero element $r\in
    \ker (\varphi)$. By hypothesis there is some $T$-ideal
     $\mathcal I$ not containing~$r$, for which $R/\mathcal I$ is
     Hopfian, but
Lemma~\ref{easierlem} implies that $R/\mathcal I$ is not
     Hopfian, a contradiction.\end{proof}

We are ready to treat the free algebra.

\begin{theorem} [\cite{OR}] When $X$ is a finite set of noncommuting indeterminates, the free associative algebra $F\{ X\} $ is Hopfian.
  \end{theorem}
\begin{proof}   Let $\varphi: F\{ X\} \to F\{ X\}$ be an epimorphism, with some nonzero polynomial $f\in
    \ker (\varphi)$. Let $n=\deg(f).$
    Let  $\mathcal I_n$ be the $T$-ideal of identities of the algebra of generic $n\times
    n$ matrices. Then $\cap \mathcal I_n = 0,$ so $F\{ X\} $ is  $T$-residually
    Hopfian, and we apply Proposition~\ref{resH}.\end{proof}

\section{Appendix: The Jacobian conjecture and the free metabelean algebra}

Dicks and Lewin~\cite[Proposition~3.1]{DL} proved that an
endomorphism of  the free associative algebra $F\{ X\} $ is an
epimorphism iff its Jacobian matrix is invertible. Their proof was
utilized by Umirbaev to obtain the following result:

\begin{theorem}[\cite{U1}]
If any endomorphism $\varphi$ of the polynomial ring with Jacobian
$1$ could be lifted to some endomorphism of the free metabelean
algebra with invertible Jacobi matrix, then the Jacobian Conjecture
would hold.
  \end{theorem}

\begin{remark} The free meta-abelian algebra was used by Umibaev~\cite{U2},  and Drensky and
Yu \cite{DY}  to prove the Anick and strong Anick conjectures.
\end{remark}

Also see  \cite{S} for a treatment of the Jacobian conjecture over a
free algebra, using his deep theory of skew fields, and
 \cite{BBRY} for an overview of Yagzev's method to attack
the Jacobian conjecture.

\end{document}